\documentclass[a4paper,oneside,english,reqno]{amsart}

\usepackage[utf8]{inputenc}
\usepackage[T1]{fontenc}

\DeclareFontFamily{OMX}{mlmex}{}
\DeclareFontShape{OMX}{mlmex}{m}{n}{%
   <->mlmex10%
   }{}%
\usepackage{mlmodern}

\usepackage[hscale=0.6, vscale=0.7]{geometry}
\usepackage{babel}
\usepackage[numbers,sort&compress]{natbib}
\usepackage{hyperref}
\hypersetup{%
  colorlinks=true,%
  citecolor=[RGB]{120,29,126},%
  pdfauthor={Jean-François Burnol},%
  pdfsubject={Zeta values},%
  pdfstartview=FitH,%
  pdfpagemode=UseNone,%
}

\usepackage{ifpdf}
\makeatletter
\ifpdf\else\@ifpackagewith{hyperref}{dvipdfmx}{}{\usepackage{breakurl}}\fi
\makeatother

\usepackage{amsmath,amsthm,amssymb}
\usepackage{mathtools}

\DeclarePairedDelimiterX\Iffint[2]{\lbrack\!\lbrack}{\rbrack\!\rbrack}{#1,#2}
\DeclarePairedDelimiterX\Ioo[2]{\lparen}{\rparen}{#1,#2}
\DeclarePairedDelimiterX\Iof[2]{\lparen}{\rbrack}{#1,#2}
\DeclarePairedDelimiterX\Ifo[2]{\lbrack}{\rparen}{#1,#2}
\DeclarePairedDelimiterX\Iff[2]{\lbrack}{\rbrack}{#1,#2}
\DeclareMathOperator{\Res}{Res}

\newcommand\NN{\mathbb{N}}
\newcommand\ZZ{\mathbb{Z}}
\newcommand\RR{\mathbb{R}}
\newcommand\CC{\mathbb{C}}

\def\e{\mathsf{e}}
\def\dt{\mathrm{d}t}

\def\dx{\mathrm{d}x}

\def\dmu{\mathrm{d}\mskip-1mu\mu}

\theoremstyle{plain}
\newtheorem{theo}{Theorem}[section]
\newtheorem{prop}[theo]{Proposition}
\newtheorem{lem}[theo]{Lemma}
\newtheorem{cor}[theo]{Corollary}

\theoremstyle{definition}

\newtheorem{rema}{Remark}[section]

\allowdisplaybreaks

\title[Dirichlet series with missing digits]{%
  On the analytic continuation of Dirichlet
  series with missing digits}

\author[J.-F. Burnol]{Jean-François Burnol}

\address{Université de Lille,
  Faculté des Sciences et technologies,
  Département de mathématiques,
  Cité Scientifique,
  F-59655 Villeneuve d'Ascq cedex,
  France}
\email{jean-francois.burnol@univ-lille.fr}

\date{February 24, 2026.}

\subjclass[2020]{11M41 (Primary) 11A63, 11B85, 11M06, 11Y35, 33F05, 68R15 (Secondary)}
\keywords{Dirichlet series with missing digits, Riemann zeta function, meromorphic continuation, residues and poles, generating functions, Cantor sets, Bernoulli numbers}

\usepackage[np,autolanguage]{numprint}

\newcommand\arxivurl[1]{\href{https://arxiv.org/abs/#1}{\textsf{arXiv:#1}}}

\usepackage{parskip}
\begin{document}

\begin{abstract}
  We study the Dirichlet series associated with the integers whose radix-$b$
  representation misses certain (fixed) digits. The existence of a meromorphic
  continuation to the entire complex plane, which was already well-known as a
  general fact valid for $b$-automatic Dirichlet series, is proven anew from a
  representation as an everywhere defined series with good convergence
  properties. A generating function related to the residues on the real axis
  is shown to be the multiplicative inverse of the moment generating function
  for the associated Cantor set in the unit interval.  This makes the
  (normalized) residues some sort of generalized Bernoulli numbers.
\end{abstract}

\maketitle


\section{Introduction}

Throughout the paper, $\NN$ is the set of non-negative integers and $\ZZ$ the
set of relative integers.  Let $b>1$ be an integer and $A\subset\Iffint0{b-1}$
be a set of digits in base $b$, containing at least one non-zero digit.  Let
$N$ be the cardinality of $A$.

This paper is devoted to a study of some analytical
properties of the ``restricted'' Dirichlet series:
\begin{equation}\label{eq:Kdef}
  K(s) = \sideset{}{'_{n>0}}\sum n^{-s}.
\end{equation}
Here, the symbol next to the summation sign means that only those integers are
kept whose $b$-ary digits are all in $A$.  We say that such integers are
``admissible'' (and the digits in $A$ are the ``admissible digits''). The
series should arguably be denoted $K_{b,A}(s)$ but we shorten the notation for
convenience.  If $N=b$, i.e.\@ if every digit is admissible, then $K(s)$ is
the defining series of the Riemann zeta function $\zeta(s)$.

This series is a member of the much wider class of \emph{$b$-automatic}
Dirichlet series.  To explain this, rewrite $K(s)$ as $\sum_{n>0} c_nn^{-s}$
where $c_n=1$ if $n$ is admissible, $c_n=0$ otherwise (and we set
$c_0=1$). The sequence $(c_n)_{n\geq0}$ is $b$-automatic (see
\cite{alloucheshallitCUP2003,allouchemendespeyriere2000} for this notion; see
also the recent paper \cite{alloshalstip2025} for a contemporary discussion of
some aspects of Dirichlet series for $b$-automatic and also for the more
general $b$-regular sequences).  As such, thanks to a general theorem of
Allouche, Mendès~France and Peyrière (\cite{allouchemendespeyriere2000}; see
also \cite[Thm.\@ 3.1]{coons2010} for $b$-regular sequences), $K(s)$ has a
meromorphic extension to the entire complex plane, with candidate poles
located in a finite union of left semi-lattices having $1$ and $2\pi i/\log b$
as periods.

The initial motivation for this paper came from another direction.  In
\cite{burnolzeta} we represented $K(s)$ in its half-plane of convergence $\Re
s > \log_b N$ in terms of a geometrically convergent series (even in the case
of the Riemann zeta function, such representation appears to be new),
involving certain coefficients $(u_m(s))_{m\geq0}$ which obey a finite
recurrence which is stable numerically.  They where first introduced for $s=1$
in \cite{burnolkempner}, where a novel algorithm was developed for the
numerical evaluation of $K(1)$ (the Kempner series
\cite{kempner,baillie1979,fischer,burnolkempner}).

In the first section of this paper, we prove that the series of
\cite{burnolzeta} makes sense everywhere in the complex plane and that its
terms (hence its remainders) are bounded by geometrically decreasing
sequences, uniformly on compact sets (see Theorem \ref{thm:main} and Corollary
\ref{cor:main} for the precise statement; a multiplicative factor is present
to cancel the poles).  This applies in particular to $\zeta(s)$, for which we
had proven already the extension to $\Re s>0$ (\cite{burnoleta}).  A new
technique is needed for the extension to the whole plane, and it is presented
here in complete generality of an arbitrary set $A$ of admissible $b$-ary
digits.

The case $N=1$ is very explicit and elementary and we discuss it briefly in
the second section.  We return next to the study with general $N$, and give in
the third section a second proof of meromorphic continuation.  We simply
follow the same steps as in the far more general theorem of
\cite{allouchemendespeyriere2000} regarding $b$-automatic Dirichlet series.
Their method, and our adaptation--instantiation, is based
upon an \emph{infinite functional equation} expressing the value at $s$ in
terms of the values at $s+1$, $s+2$, $s+3$, \dots{} .

Regarding the values at the non-positive integers, it will turn out that
$K(-m)$ is zero for all $m>0$ if $0$ is an admissible digit.  In the
alternative, these values are less trivial and we compute their exponential
generating function (this is done in the last section of this paper,
Proposition \ref{prop:zerovalues}).

Regarding the residues, we prove that, with $s_0 = \log_b N$ the abscissa of
convergence, there is a pole at $s_{m,k}=s_0 - m + k 2\pi i/\log(b)$,
$m\in\NN$, $k\in \ZZ$, if and only if both $s_{m,0}$ and $s_{0,k}$ are indeed
poles: this follows from a proportionality of rows of (modified) residues,
Proposition \ref{prop:mukmu0}. We give a formula for the residue
$\lambda_{0,k}$ at $s_{0,k}$ as a limit (Proposition \ref{prop:res}), but do
not study them further, preferring to focus on the residues $\lambda_{m,0}$ at
the $s_0-m$ locations on the real axis. It proves convenient to replace them
with some closely related quantities $\mu_{m,0}$ (see Equation
\eqref{eq:defmu}).  We can define the latter quantities also in the case of
the Riemann zeta function (for which the residues would be zero), and they
turn out to be the Bernoulli numbers.  Each choice of the integer $b>1$ gives
its own linear recurrence characterizing them.

To conclude this summary, we mention a result which we find rather striking
(from a ``structural'' point of view; we do not engage into any advanced
analytics, though).

Consider the Cantor set of the numbers in $\Iff01$ whose (infinite) $b$-radix
representation (or one of them, if there are two) uses only admissible
digits (if $N=1$ it is reduced to a singleton and if $N=b$ it is the
un-modified interval itself). It is well-known that there is a canonical
probability measure $\mu$ on this set, singular to Lebesgue measure if
$N<b$, whose moment generating function $\int_{\Iff01} \e^{tx}\dmu(x)$ is the
infinite product (which is an entire function):
\begin{equation}\label{eq:cantorcar}
  E(t) = \prod_{j=1}^\infty \frac{\sum_{a\in A} \e^{b^{-j}at}}N\;.
\end{equation}
In the case with $N=b$, a simple computation shows that the above infinite
product, a priori depending on $b$, in fact indeed computes $(\e^t-1)/t =
\int_0^1\e^{tx}\dx$.

Let us define $B(t)$ (for Bernoulli) as the \emph{multiplicative inverse} of
$E(t)$:
\begin{equation}\label{eq:B}
  B(t) = \prod_{j=1}^\infty \frac N{\sum_{a\in A} \e^{b^{-j}at}}\;.
\end{equation}
Theorem \ref{thm:E} states that \emph{$B(t)$ is the exponential generating
  function for the (modified) residues of $K(s)$ at the $s_0-m$ points,
  $m\geq0$.} In the case with $N=b$, this simply means that our ad hoc choice
for the $\mu_{m,0}$ gives the sequence of Bernoulli numbers.  So this makes
the (normalized) residues in general a kind of generalization of Bernoulli
numbers.  Probably some authors have already envisioned \eqref{eq:B} in this
perspective, but we feel that the connection with the residues of the
restricted Dirichlet series contributes to strengthen the analogy.


It would be interesting to check to what extent, some, or all, of our findings
are examples of general facts for $b$-automatic Dirichlet series, or perhaps
even more general ones.  The literature on Dirichlet series, and the one on
automatic sequences as well, are too large for the author to be able to
provide here a curated list of the various works which could prove relevant to
such endeavor.  Thus, in addition to the papers already cited we mention here
only a few, mostly related with the study of the sum-of-digits function or
other summatory functions: \cite{flajoletetal1994, coons2010,
  knilllesieutre2012, everlove2022}.

\section{Representation by geometrically convergent series}

We let $A_1 = A\setminus\{0\}$, we already defined $N=\# A$, and we set
$N_1=\#A_1 \geq1$.  The \emph{length} $\ell(n)$ of a natural integer $n\in\NN$
is the number of digits of the minimal $b$-representation, so $\ell(0)=0$
(as the integer $0$ is represented by the empty word) and $\ell(n)$ for
$n>0$ is the smallest integer exponent $l$ such that $n<b^l$.  Define for each
positive integer $l$ the finite sum (``blocks''):
\begin{equation}
  \label{eq:Kl}
  K_l(s) = \sideset{}{'_{\ell(n)=l}}\sum n^{-s}.
\end{equation}
We also define $K_{\leq \ell}(s) = \sum_{1\leq l \leq \ell} K_l(s)$, and $K_{\geq
  \ell}(s) = \sum_{l=\ell}^\infty K_l(s)$ (in case of convergence, at first),
and similarly with other types of restrictions on the indices.

The cardinality of the set of admissible integers of length $l$ is
$N_1N^{l-1}$.  So, for $s=\sigma$ real, there holds
$N_1N^{l-1}b^{-l\sigma}<K_l(\sigma)\leq N_1 N^{l-1}b^{-(l-1)\sigma}$.  Thus, the
necessary and sufficient condition for the convergence of $K(\sigma) =
\sum_{l=1}^\infty K_l(\sigma)$ is $b^\sigma>N$, i.e.\@ $\sigma>s_0 = \log_b
N$.

A general basic result (Jensen lemma, based upon a partial summation, see
\cite[Chap.\@ 4]{queffqueff2020}), which underlies the notion of abscissa of
convergence, says that as a consequence, the Dirichlet series \eqref{eq:Kdef}
defining $K(s)$ diverges whenever $\Re(s)<s_0$.  The abscissa of convergence
is thus $s_0$ (\cite{kohlspil2009}; see also \cite{nathanson2021jnt} for more
generality, and \cite{alloshalstip2025} for a still wider perspective).
Let us mention this complement:
\begin{prop}
  For any $s = s_0 + it$, $t \in \RR$, the Dirichlet series $K(s)$ diverges.
\end{prop}
\begin{proof}
  Suppose to the contrary that it converges for some $s=s_0+it$.  Let
  $n_1<n_2<\dots<n_q<\dots$ be the positive admissible integers.  Let $r_p =
  \sum_{q=p}^\infty n_q^{-s} = o(1)$. We do a summation by parts:
  \begin{equation*}
    S_p\coloneq \sum_{q=1}^p n_q^{-s_0} = \sum_{q=1}^p(r_q - r_{q+1})n_q^{it} = r_1 n_1^{it} + \sum_{q=2}^{p}r_q\bigl(n_{q}^{it}-n_{q-1}^{it}\bigr) -r_{p+1} n_p^{it}.
  \end{equation*}
  From $n_q^{it}-n_{q-1}^{it}=\int_{n_{q-1}}^{n_q} it x^{it-1}\dx$, or from
  $|\e^{i\theta}-1|\leq|\theta|$, we get $\big|n_q^{it}-n_{q-1}^{it}\bigr|
  \leq |t|\log\frac{n_q}{n_{q-1}}$. The summation on the right-hand side,
  restricted to the indices $q$ such that $p_0<q\leq p$ with some fixed $p_0$
  is thus bounded by $|t|(\sup_{q>p_0}|r_q|)\log(n_pn_{p_0}^{-1})$.  Bounding
  trivially the finitely many other terms, it follows:
    \begin{equation*}
      \lim_{p\to\infty} \frac{S_p}{\log n_p} = 0.
    \end{equation*}
    But this is wrong.  To see it, take $\ell\geq 1$ and $p_\ell= N_1 + N_1
    N + \dots + N_1 N^{\ell-1}$ so that $n_{p_\ell}$ is the largest
    admissible integer which is $<b^\ell$.  Thus $n_{p_l}< b^{\ell}$. And
    $S_{p_\ell} > N_1b^{-s_0} + N_1N b^{-2s_0} + N_1 N^2 b^{-3s_0} + \dots + N_1
    N^{\ell-1}b^{-\ell s_0} = \ell N_1N^{-1}$. Hence, $S_{p_\ell}$ is not $o(\log
    n_{p_\ell})$.
\end{proof}

In \cite{burnolzeta}, the following quantities are defined for $\Re s > s_0$,
and $m\geq0$:
\begin{equation}\label{eq:defu}
  u_m(s) = 0^m  + 
         \sum_{l=1}^\infty
         \sum_{(a_1, \dots, a_l)\in A^l}\bigl(a_1b^{-1} + \dots + a_l b^{-l}\bigr)^m\, b^{-ls}. 
\end{equation}
They are the moments of a certain complex measure on the set of $b$-words, but
we will not use this vocabulary here. The defining series for $u_m(s)$
converges if and only if it converges absolutely if and only if $\Re s > s_0$
(use that the sum of powers is bounded below by $N_1N^{l-1}b^{-m}$ and above
by $N^l$).  One has
\begin{equation*}
  u_0(s) = \frac{b^s}{b^s - N}\,,
\end{equation*}
and the recurrence
relation (\cite[Prop.\@ 1]{burnolzeta}):
\begin{equation}\label{eq:recum}
  m\geq1\implies
  (b^{s+m} - N)u_m(s) = \sum_{j=1}^{m} \binom{m}{j}(\sum_{a\in A}a^j) u_{m-j}(s),
\end{equation}
which will be the basis for our further investigations in this section.

We use the following notation for power sums
over the admissible digits:
\begin{equation*}
  \gamma_j  = \sum_{a\in A} a^j.
\end{equation*}
In particular, $\gamma_0 = N$.  We proved (\cite[Thm.\@ 1]{burnolzeta}) the
following representation of $K(s)$ in its half-plane of convergence.  For any
choice of a \emph{level} $\ell\geq 1$, and any $s$ with $\Re s > s_0$:
\begin{equation}
  \label{eq:Kgeo}
      K(s) = \sideset{}{'_{0<n<b^{\ell-1}}}\sum \frac1{n^s}
      + \sum_{m=0}^\infty (-1)^m \frac{(s)_{m}}{m!}u_{m}(s)
              \sideset{}{'_{b^{\ell-1}\leq n<b^{\ell}}}\sum \frac1{n^{m+s}}.
\end{equation}
We will simply use $\ell=2$ in the sequel (geometric convergence may not apply
if $\ell=1$, see \cite[Thm.\@ 1]{burnolzeta} for details).  The notation
$(s)_m$ is the Pochhammer symbol (partial ascending factorial with $m$ terms).

From \eqref{eq:recum} and $u_0(s) = b^s/(b^s - N)$, it is clear that all the
$u_m(s)$ extend meromorphically to the entire complex plane, with simple
\emph{candidate} poles at the points of the left semi-lattice:
\begin{equation*}
  \Delta = \{s_{m,k} = s_0 - m + k\frac{2\pi i}{\log b}, m\in \NN, k\in \ZZ\}.
\end{equation*}
In the case with $N=b$ (and also $N=1$, see the next section) there are
explicit alternative formulas for the $u_m(s)$'s.  If $N=b$ one sees that they
have poles only for $\Re s\in \{1, 0, -1, -3, -5, \dots\}$. The identical
poles coming from multiple $u_m(s)$'s conspire to cancel out, and in the end
only the pole at $s=1$ remains for $\zeta(s)$ itself...

Our aim in this section is to prove that the series at the right-hand side of
\eqref{eq:Kgeo} converges everywhere in $\CC\setminus\Delta$, thus providing
the analytic continuation of $K(s)$. We will prove this in a slightly stronger
form, so that it will be shown that $K(s)$ is meromorphic with only simple
poles.  The proof is elementary but needs some notation and care. The key
ingredient is the following Lemma:
\begin{lem}\label{lem:key}
  Let $p$ be a positive integer.  Let $0<\eta$. Let $(c_m)_{m\geq0}$ be the
  sequence of positive real numbers, depending upon $p$ and $\eta$, and
  defined by the conditions:
  \begin{gather*}
    c_0 = c_1 = \dots = c_{p-1} = 1\\
   m\geq p\implies (b^{1+\eta+m-p} - N)c_m =
   \sum_{j=1}^m\binom{m}{j}\gamma_j c_{m-j}.
  \end{gather*}
  Let $f=\max A$ and $\lambda = f/(b-1)$.  There exists an (explicit)
  continuous function $C_p$ on the positive real numbers such that $c_m \leq
  C_p(\eta)\lambda^m m^{p-1-\eta}$ for every $m\geq p$.
\end{lem}
\begin{rema}
  For $N<b$ we could take $\eta=0$, as the only problem
  with $\eta=0$ is with the definition of $c_p$ if $b-N=0$.
\end{rema}
\begin{proof}[Proof of Lemma \ref{lem:key}]
  We let $\tau=1+\eta$.
  Let $n\geq0$ and $m=n+p$. We split the binomial sum on the right-hand side of
  the recurrence as $A_n + B_n$ where $B_n$ stands for the contributions of
  the last $p$ terms.  We will take care of $B_n$ later.
  The recurrence relation is:
  \begin{equation*}
    (b^{\tau+n}-N)c_{p+n} = \sum_{j=1}^n \binom{p+n}{j}\gamma_j c_{p+n-j} + B_n.
  \end{equation*}
  Let us now make the definition, for $n\geq0$:
  \begin{equation*}
    d_n =\frac{(\tau+1)_{n}}{(p+n)!} c_{p+n}
 = \frac{(\tau+1)\dots(\tau+n)}{(p+n)!}c_{p+n}.
  \end{equation*}
  We compute:
  \begin{align*}
(b^{\tau+n}-N)d_{n} &= 
    \frac{(\tau+1)_n}{(p+n)!} A_n +  \frac{(\tau+1)_n}{(p+n)!} B_n \\
&=
    \sum_{j=1}^n \frac{(p+n)!}{j!(p+n-j)!}\gamma_j \frac{(\tau+1)_n}{(p+n)!}
                \frac{(p+n-j)!}{(\tau+1)_{n-j}}d_{n-j} 
+  \frac{(\tau+1)_n}{(p+n)!} B_n \\
&=
    \sum_{j=1}^n \frac{(\tau+n)\dots(\tau+n-j+1)}{j!}\gamma_j d_{n-j}
+  \frac{(\tau+1)_n}{(p+n)!} B_n.
  \end{align*}
We claim that
  \begin{equation}\label{eq:claim}
    \sum_{j=1}^n \frac{(\tau+n)\dots(\tau+n-j+1)}{j!}\gamma_j \lambda^{-j}
\leq b^{\tau+n}-N.
  \end{equation}
  The argument (for which $\tau\geq1$ is important) was given in \cite[Proof
  of Prop.\@ 5]{burnolzeta}.  For the convenience of the reader, we repeat it
  partially here.  First observe that $\gamma_j$ can be written as $\sum_{a\in
    A_1} a^j$ as $j\geq1$.  We have to evaluate:
 \begin{equation*}
\begin{split}   \sum_{j=1}^n \frac{(\tau+n)\dots(\tau+n-j+1)}{j!}
         \sum_{a\in A_1}\bigl(\frac{(b-1)a}{f}\bigr)^{j}
\\\leq 
   \sum_{j=1}^n \frac{(\tau+n)\dots(\tau+n-j+1)}{j!}
         \sum_{a=f,f-1,\dots,f-N_1+1}\bigl(\frac{(b-1)a}{f}\bigr)^{j}.
\end{split}
 \end{equation*}
 The following calculus lemma (\cite[Lem.\@ 3]{burnolzeta}) holds: \emph{Let
   $\sigma\geq1$ be a real number, $n$ a non-negative integer, and $x$ a
   non-negative real number.  Then}
\begin{equation*}
  \sum_{j=1}^n \frac{(n+\sigma)\dots(n+\sigma-j+1)}{j!} 
              x^j \leq (1+ x)^{n+\sigma}- 1 - x^{n+\sigma}.
\end{equation*}
Thus, if $b-1=f$ we obtain directly a telescopic sum as upper bound. If
$b-1>f$, one needs only make a simple additional remark to again reduce to a
telescopic sum, see \cite[Proof of Prop.\@ 5]{burnolzeta}. This gives us:
\begin{equation*}
  \textit{currently evaluated sum}\leq b^{\tau+n}-N_1 -1 \leq b^{\tau+n} - N,
\end{equation*}
which validates claim \eqref{eq:claim}.  Let now, for $k\geq0$,
$D_k=\max_{0\leq j\leq k} \lambda^{-j}d_{j}$ (in particular, $D_0=d_0=\frac{c_p}{p!}$),
so that $d_{n-j}\leq \lambda^{n}\lambda^{-j}D_{n-1}$ for $1\leq j \leq n$.  We
get, using \eqref{eq:claim}:
\begin{align*}
  (b^{n+\tau}-N)d_n &\leq \lambda^n( b^{\tau+n} - N) D_{n-1}
+  \frac{(\tau+1)_n}{(p+n)!} B_n\\
\lambda^{-n}d_n &\leq D_{n-1} 
+ \frac{(\tau+1)_n}{(p+n)!} \frac{\lambda^{-n}B_n}{b^{n+\tau}-N}\\
  D_n &\leq D_{n-1}
+ \frac{(\tau+1)_n}{(p+n)!} \frac{\lambda^{-n}B_n}{b^{n+\tau}-N}
\leq \sum_{n=0}^\infty \frac{(\tau+1)_n}{(p+n)!} 
         \frac{\lambda^{-n}B_n}{b^{n+\tau}-N}.
\end{align*}
(The $n=0$ term is indeed $D_0=d_0$).  Let $D^p(\tau)$ be the series on the
right-hand side. We show that it is a continuous function of $\tau>0$.  There
holds (replacing in the last step $(n+p-k)!$ by $(n+1)!$):
  \begin{equation*}
    B_n =  \sum_{j=n+1}^{n+p} \binom{n+p}{j}\gamma_j 
        = \sum_{k=0}^{p-1}\binom{n+p}{k}\gamma_{n+p-k}
        \leq N_1\sum_{k=0}^{p-1}\frac{(n+p)!}{k!(n+1)!}f^{n+p-k},
  \end{equation*}
  and it is thus enough
 to examine if the following series converges uniformly
  for $\tau$ in a compact sub-interval of $\Ioo1\infty$ (recall
  $\tau=1+\eta$):
\begin{equation*}
   \sum_{n=0}^\infty \frac{(\tau+1)_n}{(n+1)!}\frac{(b-1)^n}{b^{n+\tau}-N}.
\end{equation*}
Suppose $1<\tau_0\leq \tau\leq \tau_1<\infty$. Then the above is bounded termwise
by 
\begin{equation*}
 \sum_{n=0}^\infty \frac{(\tau_1+1)_n}{(n+1)!}\frac{(b-1)^n}{b^{n+\tau_0}-N}.
\end{equation*}
The latter series converges by the ratio test.  Consequently
$D^p(\tau)<\infty$ with locally uniform convergence, hence it is a continuous
function such that $D_n\leq D^p(\tau)$ for every $n\geq0$.  Hence, $d_n \leq
D^p(\tau)\lambda^n$ and consequently
\begin{align*}
  m\geq p \implies c_m &\leq D^p(\tau)\frac{m!}{(\tau+1)_{m-p}}\lambda^{m-p}\\
\implies c_m(m-p)^\tau m^{-p}\lambda^{-m}&\leq \lambda^{-p} D^p(\tau)\frac{(m-p)!(m-p)^{\tau}}{(\tau+1)_{m-p}}
\end{align*}
The last equation is useful only for  $m\geq p+1$.
Let now $\alpha_n= n!n^\tau/(\tau+1)_n$ for $n\geq1$.  As is well-known,
$\alpha_n$ converges toward $\Gamma(\tau+1)$, and of course this is uniform
for $\tau$ bounded.  In fact, the sequence $(\alpha_n)$ is increasing, 
and thus bounded above by $\Gamma(\tau+1)$.  To check that, i.e.\@ to prove that
\begin{equation*}
  \frac{(n+1)(n+1)^\tau n^{-\tau}}{n+1+\tau}\geq 1,
\end{equation*}
it is enough to show that the logarithmic derivative with respect to $\tau$ of
the left-hand side is non-negative on $\Ifo0\infty$, and actually it is
positive:
\begin{equation*}
  \log(n+1)-\log(n) - \frac{1}{n+1+\tau}=\int_n^{n+1}\frac{\dt}t - \frac1{n+1+\tau}>0.
\end{equation*}
So, for $m>p$ we obtain
\begin{equation*}
   c_m\leq \bigl(1-\frac pm\bigr)^{-\tau} \lambda^{-p} D^p(\tau)\Gamma(\tau+1)\lambda^mm^{p-\tau}\leq \bigl(p+1\bigr)^{\tau} \lambda^{-p} D^p(\tau)\Gamma(\tau+1)\lambda^mm^{p-\tau} 
\end{equation*}
And as $c_p \leq D^p(\tau)p!$, the above upper bound holds also for $m=p$. As
$\tau=1+\eta$, this completes the (somewhat long) proof of Lemma
\ref{lem:key}.
\end{proof}

We can now establish straightforwardly our main objective. Recall $\lambda =
f(b-1)^{-1}$, $f = \max A$.
\begin{theo}\label{thm:main}
  Let $\alpha(s)$ be the entire function $\prod_{n=0}^\infty (1 - b^{-s-n}N)$.
  One has, uniformly on every compact subset of the complex plane:
  \begin{equation*}
    \alpha(s) \frac{(s)_m}{m!} u_m(s) =_{m\to\infty} O\bigl(\lambda^m m^{-1}\bigr).
  \end{equation*}
\end{theo}
\begin{proof}
  Let $Q\subset \CC$ be some non-empty compact subset and choose $p$ a
  positive integer such that $\Re s \geq 2 - p$ for every $s\in Q$.  For any
  $s\in Q$, we let $\eta(s) = \Re s + p -1 \geq1$.  The analytic coefficients
  $v_m(s) = \alpha(s) u_m(s)$, $m\geq0$, verify the recurrence
  \begin{equation*}
    (b^{m+s} - N)v_m(s) = \sum_{j=1}^{m} \binom{m}{j}\gamma_j v_{m-j}(s).
  \end{equation*}
  The sequence of positive real numbers $(c_m)$ from Lemma \ref{lem:key}, for
  this choice of $p>0$ and $\eta=\eta(s)\geq1$, verify the recurrence,
  starting at $m=p$,
  \begin{equation*}
    m\geq p \implies (b^{m+\Re s} - N)c_m 
= \sum_{j=1}^{m} \binom{m}{j}\gamma_j c_{m-j}\,,
  \end{equation*}
  and with initial conditions $c_0 = \dots = c_{p-1} = 1$.

  Observe that $|b^{m+s} -N|\geq b^{m+\Re s} - N\geq b^2-N>0$ for any $m\geq
  p$.  Thus, as the binomial coefficients and power sums are positive, an
  induction proves that
  \begin{equation*}
    \forall m\geq0\quad  
    |v_m(s)| \leq c_m \max_{0\leq j\leq p-1}\|v_j(s)\|_\infty,
  \end{equation*}
  where the sup-norms are computed on $Q$.
  Hence, according to the conclusion of Lemma \ref{lem:key}, and setting
  \begin{equation*}
    M_Q =  \max_{0\leq j\leq p-1}\|\alpha(s)u_j(s)\|_\infty, \qquad
    M_Q' = M_Q \sup_{s\in Q} C_p(\eta(s)),
  \end{equation*}
  we get, using  $p-1-\eta(s) = -\Re s$:
  \begin{equation*}
    \forall m\geq p\quad  
    |\alpha(s)u_m(s)| \leq M_Q C_p(\eta(s))\lambda^m m^{p-1-\eta(s)} 
    \leq M_Q'\lambda^m m^{-\Re s}.
  \end{equation*}
  It is known that
  \begin{equation*}
    \frac{m^{1-s}(s)_m}{m!} \mathop{\longrightarrow}\limits_{m\to\infty} \frac1{\Gamma(s)},
  \end{equation*}
  uniformly on every compact subset of the complex plane. This completes the
  proof of the uniform $O(\lambda^m m^{-1})$ estimate for $\alpha(s)(s)_m
  (m!)^{-1}u_m(s)$.
\end{proof}
\begin{cor}\label{cor:main}
  The series, whose terms are entire functions,
    \begin{equation}
      \label{eq:foo}
      \sum_{m=0}^\infty(-1)^m\frac{(s)_m}{m!} \alpha(s) u_m(s)
          \sum_{\substack{a_1\in A_1\\[.5\jot]a_2\in A}}(a_1b+a_2)^{-s-m},
    \end{equation}
    converges everywhere in the complex plane and defines an entire function
    which gives the analytic continuation of $\alpha(s)\Bigl(K(s) -
    \sum_{a\in A_1} a^{-s}\Bigr)$.  Consequently, $K(s)$ is a meromorphic
    function in the complex plane with candidate simple poles at the zeros of
    $\alpha(s)=\prod_{n=0}^\infty (1 - b^{-s-n}N)$.
\end{cor}
\begin{proof}
  Immediate consequence of Theorem \ref{thm:main}: indeed, equation
  \eqref{eq:Kmain} tells us that for $\Re s > s_0$ the series of Equation
  \eqref{eq:foo} converges toward $\alpha(s)\Bigl(K(s) - \sum_{a\in A_1}
  a^{-s}\Bigr)$.  As it converges locally uniformly everywhere it provides the
  analytic continuation.
\end{proof}

\section{The case with \texorpdfstring{$N=1$}{N=1}}

Suppose $N=1$.  There is only
one admissible integer of length $l$, and it is $\lambda(b^l - 1)$, with
$\lambda=f(b-1)^{-1}$ where $f>0$ is the sole admissible digit.  So,
\begin{equation*}
  K(s) = \Bigl(\frac{b-1}{f}\Bigr)^s\sum_{l=1}^\infty\frac{1}{(b^l-1)^s}.
\end{equation*}
We will assume from here on in this section that $f=b-1$, as the general case reduces
to this one.  There is a well-known identity valid for $\Re s>0$ (but we could not
locate a standard reference):
\begin{equation}
  \label{eq:Lambert}
  \sum_{l=1}^\infty\frac{1}{(b^l-1)^s} = \sum_{j=0}^\infty \frac{(s)_j/j!}{b^{s+j}-1}.
\end{equation}
To prove this, one only needs by analyticity, to do it
for $s$ real and positive. It then follows from interverting summations in a
double sum with positive terms, obtained from using the binomial series.  The
right-hand side in \eqref{eq:Lambert} converges everywhere in $\CC\setminus\Delta$,
where $\Delta$ is the left semi-lattice $\{-m+k 2\pi i(\log b)^{-1}, m\in\NN,
k\in\ZZ\}$ (and at a candidate pole, removing the term where it intervenes,
we have an analytic function).

Consequently, all the $s_{m,k} = -m+k 2\pi i(\log b)^{-1}, m\in\NN, k\in\ZZ$,
are simple poles and the residues are:
\begin{equation}\label{eq:resN1}
  \lambda_{m,k} =
  \frac{(-m+k \frac{2\pi i}{\log b})(-m +1 +k \frac{2\pi i}{\log b})
              \dots(-1+k\frac{2\pi i}
  {\log b})}{m!\, \log b},
\end{equation}
For $m=0$ the numerator is to be interpreted as being $1$. For $k=0$ one has
$\lambda_{m,0} = (-1)^m(\log b)^{-1}$.

One may wonder if the series \eqref{eq:Lambert} is the one given by our Corollary
\ref{cor:main}, which states the validity everywhere of Equation
\eqref{eq:Kgeo}. The answer is no! To start with, Equation
\eqref{eq:Kgeo} has a parameter $\ell$, which can be any positive integer, and
it does a resummation of the original Dirichlet series from the left-hand side
of \eqref{eq:Lambert} only starting at $l=\ell$.

We need to explain what the coefficients $(u_m(s))_{m\geq0}$ are for this
case. Let us denote them $w_m(s)$.  From Equation \eqref{eq:recum}, they verify
the recurrence
\begin{equation*}
(b^{s+m} -1) w_m(s) = \sum_{j=1}^m \binom{m}{j}(b-1)^j w_{m-j}(s), 
\end{equation*}
with $w_0(s) = b^s(b^s-1)^{-1}$.  These coefficients are known. Actually we
only have to use formula \eqref{eq:defu} which gives the original
definition of the sequence $(u_m(s))_{m\geq0}$.  For $\Re s >0$:
\begin{equation*}
  w_m(s) = 0^m + \sum_{l=1}^\infty (1 - b^{-l})^m b^{-ls} 
= \sum_{l=0}^\infty(1-b^{-l})^m b^{-ls} 
\end{equation*}
Thus, we have the explicit finite Lambert-type sums here:
\begin{equation}
  \label{eq:wm}
  w_m(s) = \sum_{j=0}^m (-1)^j\binom{m}{j}\frac{b^{s+j}}{b^{s+j}-1} = 
0^m +  \sum_{j=0}^m (-1)^j\binom{m}{j}\frac{1}{b^{s+j}-1}.
\end{equation}

We can then enjoy exercises such as checking, for every $\ell\geq1$:
\begin{equation*}
  \sum_{1\leq l < \ell}\frac1{(b^l -1)^s} 
+  \sum_{m=0}^\infty (-1)^m \frac{(s)_m}{m!}\frac{w_m(s)}{(b^\ell-1)^{m+s}} 
=   \sum_{j=0}^\infty \frac{(s)_j/j!}{b^{s+j}-1},
\end{equation*}
which is what Equation \eqref{eq:Kgeo} and Corollary \ref{cor:main} say,
combined with \eqref{eq:Lambert}.

\section{Analytic continuation (second method)}


In this section we prove again that $K(s)=\sum'_{n>0}n^{-s}=\sum_{l=1}^\infty
K_l(s)$ admits a continuation to the whole complex plane as a meromorphic
function.  The method is the general one from
\cite{allouchemendespeyriere2000} (see also the earlier paper
\cite{allouchecohen1985}).  Our account will be very detailed, probably too much!

Two cases often require special provisions: $N=1$ and $N=b$.  They correspond
to very explicit Dirichlet series:
\begin{itemize}
\item If $N=b$, there are no restrictions and $K(s) = \zeta(s)$,
      the Riemann zeta function,
\item If $N=1$ then $K(s) =
  ((b-1)/f)^s\sum_{l=1}^\infty (b^l-1)^{-s}$ where $f>0$ is the sole admissible
  digit.  This was the topic of the previous section.
\end{itemize}
The first pole on the real axis is located for them
at an integer point $s_0$:  $s_0=1$ for $\zeta(s)$ and $s_0=0$ for $N=1$.
All other cases are with $0<s_0<1$.

As in \cite{allouchemendespeyriere2000} (where the analogous quantity is
denoted $F(s)$, whereas our $K(s)$ is analogous to their $G(s)$), it proves
convenient to also consider the variant
\begin{equation*}
  H(s)=  \sideset{}{'_{n>0}}\sum (n+1)^{-s}.
\end{equation*}
The abscissa of convergence is again $\log_b N$, the counting argument given
near the start of the first section  applying almost
identically. In the case with $A=\{b-1\}$, we simply have $H(s)=\frac{1}{b^s -
  1}$, which has poles only on the line $\Re s = 0$. But, as we have seen, $K(s)$ in
this case has poles on all lines $\Re s = -m$, $m\in \NN$.

We establish that $K(s) - H(s)$ is analytic for $\Re s>s_0 - 1$.
Consider the simple estimate (assuming $\sigma=\Re s\geq-1$):
\begin{equation*}
  \bigl|n^{-s}-(n+1)^{-s}\bigr| = \left|\int_{n}^{n+1}\frac{s\,\dt}{t^{s+1}} \right|
\leq |s|n^{-\sigma-1}.
\end{equation*}
This bound implies that $\sum'\bigl(n^{-s}-(n+1)^{-s}\bigr)$ converges to an analytic
function in the half-plane $\Re s > s_0 - 1$.  So $K(s)$ and $H(s)$ have
the same poles and residues in this half-plane.

We observe that admissible integers of length $l+1$ are of the shape $bn+a$,
with $n$ admissible of length $l$ and $a \in A$.  We now replace $(bn+a)^{-s}$
in $K_{l+1}(s)$ by its binomial series expansion in powers of $a/(bn)$. This
leads to a representation of $K_{l+1}(s)$ in terms of the $K_l(s+m)$'s, $m\geq0$.
This step is also the starting point in the Baillie algorithm
\cite{baillie1979}.
\begin{align}
\notag
(bn + a)^{-s} &=  \sum_{m=0}^\infty (-1)^m \frac{(s)_m}{m!}(bn)^{-s-m}a^m,\\
\label{eq:Klplus1}
  K_{l+1}(s) &= \sum_{m=0}^\infty (-1)^m \frac{(s)_m}{m!}b^{-s-m}\gamma_m K_l(s+m), \\
\notag
\text{\llap{with: }}  \gamma_m &= \sum_{a\in A} a^m \quad
  \text{and}\quad
(s)_m = \prod_{0\leq i <m}(s+i).
\end{align}
The Pochhammer symbol $(s)_m$ is defined to be $1$ for $m=0$ (this rule
applies to all empty products).  Except for $m=0$, $\gamma_m$ is also
$\sum_{a\in A_1} a^m$.

We now show that the series \eqref{eq:Klplus1} is absolutely convergent,
uniformly for $s=\sigma+it$ in any compact subset of $\{\Re s>\log_b
N\}$.  First we take note of the bound
\begin{equation*}
  \frac{|(s)_m|}{(\sigma)_m}\leq \frac{\Gamma(\sigma)}{|\Gamma(s)|}.
\end{equation*}
Indeed, $\Gamma(s)(s)_m=\Gamma(s+m)$ and we can use the integral
representation to get $|\Gamma(s+m)|\leq \Gamma(\sigma+m)$ if
$\sigma+m>0$. Here, $\sigma>0$ and $(\sigma)_m>0$.
Thus, the series in 
\eqref{eq:Klplus1} is bounded termwise by
\begin{equation*}
\frac{\Gamma(\sigma)}{|\Gamma(s)|}
\sum_m \frac{(\sigma)_m}{m!}b^{-\sigma-m}\gamma_m
  \mathop{\sum\nolimits'}\limits_{\ell(n)=l} n^{-\sigma-m} 
=
\frac{\Gamma(\sigma)}{|\Gamma(s)|}
\mathop{\sum\nolimits'}\limits_{\substack{\ell(n)=l\\ a \in A}}(bn - a)^{-\sigma}.
\end{equation*}
Call $\alpha_l$, say, the finite sum at the right-hand side. As $bn-a\geq
b^{l-1}$ if $n\geq b^{l-1}$, we have certainly $\alpha_l \leq
\mathop{\sum\nolimits'}\limits_{\ell(n)=l}Nb^{-(l-1)\sigma}=
N_1\cdot N^{l-1}\cdot N \cdot b^{-(l-1)\sigma}$. Hence the series $\sum_{l\geq
  l} \alpha_l$ converges, as soon as $\sigma>\log_b N$.

Hence, for $\Re s= \sigma>\log_b N$, we can sum the identity \eqref{eq:Klplus1}
from $l=1$ to infinity and intervert the summations over $l$ and over
$m$, obtaining:
\begin{equation*}
  K_{>\ell}(s) = N b^{-s}K_{\geq\ell}(s) + 
\sum_{m=1}^\infty (-1)^m\frac{(s)_m}{m!} b^{-s-m}\gamma_m K_{\geq \ell}(s+m).
\end{equation*}
Adding the block $K_{\ell}(s)$ to both sides, and moving $K_{\geq\ell}(s)$ to
the left-hand side, we obtain:
\begin{equation}
  \label{eq:Kellmain}
 \bigl(1 - N b^{-s}\bigr) K_{\geq\ell}(s) 
= K_\ell(s) + \sum_{m=1}^\infty (-1)^m\frac{(s)_m}{m!} b^{-s-m}\gamma_m K_{\geq \ell}(s+m).
\end{equation}
Equation \eqref{eq:Kellmain} is for now established only for $\Re
s=\sigma>\log_b N$.

From here on, let $s_0 = \log_b N$.

To prepare for the upcoming inductive reasoning, we observe that for any
compact set $Q$ in the complex plane, if we choose $m_0$ such that
$m_0+Q\subset \{z, \Re z > s_0\}$, the terms of the above series starting with
$m=m_0$ are defined (as evaluations are done in the half-plane $\Re s > s_0$)
and this sub-series converges absolutely and uniformly.  Indeed, we can bound
$K_{\geq \ell}(s+m)$ by $K(m_0 + \inf(\{\Re z, z\in Q\}))$, and $\gamma_m$ by $N
(b-1)^m$, and then apply the (uniform) ratio test.

In the next proposition, the cases $N=b$ (i.e.\@ when $K(s) = \zeta(s)$) and
$N=1$ both have their own specificities, but we do not handle them
here separately.
\begin{prop}\label{prop:cont}
  Let $b>1$ be an integer and $A$ be some subset of $\Iffint{0}{b-1}$, not
  empty and not reduced to $\{0\}$, of cardinality $N$.  The associated
  restricted Dirichlet series $K(s)$ admits a meromorphic continuation, still
  denoted $K(s)$, to the entire complex plane.  It verifies the identity
\begin{equation}\label{eq:Kmain}
 \bigl(1 - N b^{-s}\bigr) K(s) = \sum_{a\in A_1} a^{-s}  
 + \sum_{m=1}^\infty (-1)^m\frac{(s)_m}{m!} b^{-s-m}\gamma_m K(s+m),
\end{equation}
where $(s)_m$ is the Pochhammer symbol and the power sums $\gamma_m$ are
defined as:
\begin{equation}\label{eq:gammam}
  \gamma_m = \sum_{a\in A} a^m.
\end{equation}
The function
  \begin{equation}\label{eq:Lambda}
    \Lambda_K(s) = \prod_{m=0}^\infty (1 - N b^{-s-m})K(s),
  \end{equation}
  is entire.  The function $K(s)$ has a simple pole at $s_0=\log_b N$, with a
  positive residue $\lambda_0$.  It has only simple poles in the complex
  plane, which are among the $s_{m,k} = s_0 - m +
  k\frac{2\pi i}{\log b}$, $m\in \NN$, $k\in \ZZ$.
\end{prop}
\begin{rema}
  As we already mentioned this is only a very special instance of the general
  result from \cite{allouchemendespeyriere2000}.
\end{rema}
\begin{proof}
  Let us first suppose that $\Re s>s_0 - 1$.  So the series in
  \eqref{eq:Kmain} (which is the special case $\ell=1$ for Equation
  \eqref{eq:Kellmain}, whose validity is known for $\Re s > s_0$) converges
  and defines an analytic function in that half-plane.  Hence, $K(s)$ admits a
  meromorphic continution to it, with potential simple poles at the zeros
  $s_{0,k}$ of the equation $b^s = N$, which all have their real part equal to
  $s_0$.  As $\lim_{\sigma\to s_0^*}K(\sigma)=+\infty$, there really is a pole
  at $s_0$ and its residue $\lambda_{0,0}$ is positive.

  Suppose by induction on $m\geq1$ that we have proven analytic continuation
  as a meromorphic function with simple poles to $\Re s > s_0-m$. This is
  known for $m=1$.  The right-hand side of Equation \eqref{eq:Kmain} defines a
  meromorphic function in the half-plane $\{\Re s > s_0 - m -1\}$.  In the extended
  strip $-1+s_0-m<\Re s< s_0$, the division by $(1 - Nb^{-s})$ brings no
  extra singularity.  So this equation defines the analytic continuation, with
  simple poles located at the roots of the equations $b^s=N$, $b^{s+1}=N$,
  \dots, $b^{s+m}=N$.  By induction, this is true for every $m\geq1$.

  The infinite product $\prod_{m=0}^\infty (1 - Nb^{-s-m})$ is absolutely and
  uniformly convergent on any compact set, after having removed finitely many
  factors so that its terms have no zeros there, so it defines an entire
  function with simple zeros at the $s_{m,k}$, $m\in \NN$, $k\in \ZZ$.  Hence
  $\Lambda_K(s)$ from \eqref{eq:Lambda} is entire. And we have proven that
  \eqref{eq:Kmain} holds everywhere in the complex plane.
\end{proof}
\begin{rema}
  In the case of the Riemann zeta function, $N=b$, the statement places
  potential poles at $1 -\NN+\frac{2\pi i}{\log b} \ZZ$. Using $b=2$ and $b=3$
  and the linear independence over the rationals of $(\log 2, \log 3)$, we can
  reduce the potential pole locations to $1 - \NN$.  We shall provide another
  proof later, and only then address the urgent task of showing that $1$ is
  the only pole!
\end{rema}

Prior to engaging into a more detailed discussion of the poles and residues, we
state here the analytic continuation for $H(s) = \sum'_{n>0}(n+1)^{-s}$.
\begin{prop}\label{prop:KmainH}
  The modified restricted Dirichlet series $H(s)=\sum'_{n>0} (n+1)^{-s}$ admits a
  meromorphic continuation, still denoted $H(s)$, to the entire complex
  plane.  The function
  \begin{equation}\label{eq:LambdaH}
    \Lambda_H(s) = \prod_{m=0}^\infty (1 - N b^{-s-m})H(s).
  \end{equation}
   is entire.
  The relation
\begin{equation}\label{eq:KmainH}
 \bigl(1 - N b^{-s}\bigr) H(s) = \sum_{a\in A_1} (a+1)^{-s}  
 + \sum_{m=1}^\infty \frac{(s)_m}{m!} b^{-s-m}\gamma_m' H(s+m)
\end{equation}
holds at all $s$ such that $s\neq s_{m,k}=\log_b N - m + k\frac{2\pi i}{\log b}$,
$m\geq0$, $k\in \ZZ$.  Here the modified power sums $\gamma_m'$ are defined as
\begin{equation}\label{eq:gammamprime}
  \gamma_m' = \sum_{a\in A} (b-1-a)^m.
\end{equation}
\end{prop}
\begin{proof}
  Let $l\geq1$.  Let $n_1$ be an admissible integer of length $l+1$, $n_1 = bn
  + a$, $a\in A$. We use the decompostion
\begin{equation*}
  n_1 + 1 = bn + a + 1 = b(n+1) - (b-1 - a),
\end{equation*}
and binomial series to compute $(n_1+1)^{-s}$.  Summing over the
admissible digits $a\in A$ will be expressed via the variant power sums
$\gamma_m'$.  The power series is the one of $(1 - t)^{-s}$ which absorbs the
$(-1)^m$ arising in $(1+t)^{-s}$.  The reasoning leading to \eqref{eq:KmainH}
is then almost identical to the one which gave \eqref{eq:Kmain}. Details left
to the reader.
\end{proof}


All contents of this section are but an illustration of long since known facts
\cite{allouchemendespeyriere2000}.

\section{Residues (I)}

Let us abridge the residue $\lambda_{0,0}$ at $s_0$ to $\lambda_0$.  As a
consequence of \eqref{eq:Kmain}, we obtain
\begin{equation}\label{eq:res}
  \lambda_0 \log b = 
 \sum_{a\in A_1} a^{-s_0} +   \sum_{m=1}^\infty (-1)^m\frac{(s_0)_m}{m!} N^{-1}b^{-m}\gamma_m K(s_0+m).
\end{equation}
The above equation does not make it particularly easy to check the positivity,
but it is possible.  If $s_0=0$ (i.e.\@ $N=1$), the series above vanishes
identically, the set $A_1$ is a singleton $\{a\}$ and $\lambda_0 \log b =
1$. We suppose $s_0>0$ from now on for this discussion.  Let us first observe
that (as is well-known) the individual series:
\begin{equation*}
  (bn+a)^{-s_0}-(bn)^{-s_0} = 
  \sum_{m=1}^\infty (-1)^m\frac{(s_0)_m}{m!}\frac{a^m}{(bn)^{s_0+m}}
\end{equation*}
have the property that the remainders have the same sign as their first term
(as $s_0>0$). This property is inherited by the
series from \eqref{eq:res}.  So, as $s_0>0$, there holds
\begin{equation*}
  \lambda_0 \log b > \sum_{a\in A_1} a^{-s_0} - \frac{s_0 \gamma_1}{Nb}K(s_0+1).
\end{equation*}
The same principle applied at $s=s_0+1$ in \eqref{eq:Kmain},
with $\ell=1$ gives
\begin{equation*}
  (1 - b^{-1})K(s_0+1) < \sum_{a\in A_1}a^{-s_0-1}.
\end{equation*}
Combining the two estimates we get
\begin{equation*}
  \lambda_0\log b > \sum_{a\in A_1} a^{-s_0} -  \frac{s_0 \gamma_1}{Nb}\frac1{1 - b^{-1}}
 \sum_{a\in A_1}a^{-s_0-1}.
\end{equation*}
We delay the discussion of the Riemann zeta function to the next section, and
assume $s_0<1$ for now. There holds $\gamma_1\leq N(b-1)$, so we obtain as
lower bound
\begin{equation*}
  \lambda_0\log b> \sum_{a\in A_1} (1-s_0 a^{-1})a^{-s_0}>0.
\end{equation*}
This confirms that the residue is positive.

A simpler method is to start from the formula for $\lambda_0$ which follows
from Equation \eqref{eq:KmainH} (recall the earlier observation that $K$ and $H$
have the same poles and residues for $\Re s = s_0$). We obtain:
\begin{equation*}
  \lambda_0 \log b =  \sum_{a\in A_1} (a+1)^{-s_0}
 + \sum_{m=1}^\infty \frac{(s_0)_m}{m!} N^{-1}b^{-m}\gamma_m' H(s_0+m),
\end{equation*}
which is visibly positive.
More generally there holds, for $\lambda_{0,k} = \Res_{s_{0,k}} K(s)$:
\begin{equation}\label{eq:resH}
  \lambda_{0,k} \log b =  \sum_{a\in A_1} (a+1)^{-s_{0,k}}  
 + \sum_{m=1}^\infty \frac{(s_{0,k})_m}{m!} N^{-1}b^{-m}\gamma_m' H(s_{0,k}+m).
\end{equation}

We reverse-engineer the binomial series in the
following manner, for any $s$ with $\Re s = \sigma> s_0-1$.  Let us first
consider this doubly infinite sum with non-negative terms:
\begin{align*}
  &\sum_{m=1}^\infty \sum_{l=1}^\infty \sum_{a\in A}
   \frac{|(s)_m|}{m!} b^{-\sigma-m}
    \mathop{\sum\nolimits'}\limits_{\ell(n)=l}\frac{(b-1-a)^m}{(n+1)^{\sigma+m}}
\\
  &\leq\frac{\Gamma(\sigma+1)}{|\Gamma(s+1)|}
   \sum_{m=1}^\infty \sum_{l=1}^\infty \sum_{a\in A}
   \frac{|s|}{|s+m|}\frac{(\sigma+1)_m}{m!} b^{-\sigma-m}
    \mathop{\sum\nolimits'}\limits_{\ell(n)=l}\frac{(b-1-a)^m}{(n+1)^{\sigma+m}}
\\
  &\leq\frac{\Gamma(\sigma+1)}{|\Gamma(s+1)|}
   \sum_{l=1}^\infty \mathop{\sum\nolimits'}\limits_{\ell(n)=l}\sum_{a\in A}
   (bn+b)^{-\sigma}
    \biggl(\Bigl(1 - \frac{b-1-a}{bn+b}\Bigr)^{-\sigma-1} - 1\biggr)
\\
  &\leq\frac{\Gamma(\sigma+1)}{|\Gamma(s+1)|}
   \sum_{l=1}^\infty \mathop{\sum\nolimits'}\limits_{\ell(n)=l}\sum_{a\in A}
   (bn+b)
    \bigl((bn+a+1)^{-\sigma-1} - (bn+b)^{-\sigma-1}\bigr).
\\
  &\leq Nb^{-\sigma}\frac{\Gamma(\sigma+1)}{|\Gamma(s+1)|}
   \sum_{l=1}^\infty \mathop{\sum\nolimits'}\limits_{\ell(n)=l}
   (n+1)
    \bigl(n^{-\sigma-1} - (n+1)^{-\sigma-1}\bigr)
\\
  &\leq Nb^{-\sigma}(\sigma+1)\frac{\Gamma(\sigma+1)}{|\Gamma(s+1)|}
   \sum_{l=1}^\infty \mathop{\sum\nolimits'}\limits_{\ell(n)=l}
   (n+1)n^{-\sigma-2} < \infty.
\end{align*}
Thus, it is finite.  Hence the next computations are justified as soon as $\Re
s> s_0-1$:
\begin{align*}
  &\sum_{m=1}^\infty \frac{(s)_m}{m!} b^{-s}b^{-m}\gamma_m' H(s+m)
\\
&=\sum_{m=1}^\infty \sum_{l=1}^\infty \sum_{a\in A}
   \frac{(s)_m}{m!} b^{-s-m}
    \mathop{\sum\nolimits'}\limits_{\ell(n)=l}\frac{(b-1-a)^m}{(n+1)^{s+m}}
\\
&=\sum_{l=1}^\infty \mathop{\sum\nolimits'}\limits_{\ell(n)=l}
  \sum_{a\in A}
  (bn+b)^{-s}\left(\bigl(1 - \frac{b-1-a}{bn+b}\bigr)^{-s} - 1\right)
\\
&=\lim_{\ell\to\infty}\sum_{l=1}^\ell \mathop{\sum\nolimits'}\limits_{\ell(n)=l}
  \sum_{a\in A}
  \left((bn+a+1)^{-s} - (bn+b)^{-s}\right)
\end{align*}
Something special happens if $b^s=N$, i.e.\@ if $s=s_{0,k}$, $k\in \ZZ$, as
then $\sum_{a\in A} b^{-s} = Nb^{-s} = 1$.  So the computation gives:
\begin{align*}
&\sum_{m=1}^\infty \frac{(s)_m}{m!} N^{-1}b^{-m}\gamma_m' H(s+m)
\\
&=\lim_{\ell\to\infty}\sum_{l=1}^\ell \mathop{\sum\nolimits'}\limits_{\ell(n)=l}
  \left(\Bigl(\sum_{a\in A} (bn+a+1)^{-s_{0,k}}\Bigr) - (n+1)^{-s_{0,k}}\right)
\\
&=-\sum_{a\in A_1} (a+1)^{-s_{0,k}} + \lim_{\ell\to\infty}\mathop{\sum\nolimits'}\limits_{\ell(n)=\ell+1} (n+1)^{-s_{0,k}}.
\end{align*}
Combining with Equation \eqref{eq:resH} leads to the next proposition:
\begin{prop}\label{prop:res}
  The residues $\lambda_{0,k}=\Res_{s_{0,k}}K(s)$ verify:
  \begin{equation}\label{eq:reslim}
    \lambda_{0,k} = (\log b)^{-1}\lim_{\ell\to\infty} 
    \mathop{\sum\nolimits'}\limits_{\ell(n)=\ell} \frac1{n^{s_{0,k}}}.
  \end{equation}
\end{prop}
\begin{proof}
  The chain of equations prior to the proposition, and leading to a telescopic
  series, proves from \eqref{eq:resH} that the residue of $H(s)$ at
  $s=s_{0,k}$ is:
  \begin{equation*}
    (\log b)^{-1}\lim_{\ell\to\infty} 
    \mathop{\sum\nolimits'}\limits_{\ell(n)=\ell} \frac1{(n+1)^{s_{0,k}}}.
  \end{equation*}
  There only remains to explain why we can replace $n+1$ above with $n$. We
  use again the simple-minded upper bound $|n^{-s}-(n+1)^{-s}|\leq |s|
  n^{-\sigma-1}$, valid for $\sigma+1\geq0$, in particular it is valid with $s
  = s_{0,k}$ for which $\sigma=s_0$. And as $\sum' n^{-s_0-1}<\infty$, the
  contribution of the admissible integers of length $\ell$ goes to zero as
  $\ell$ goes to $\infty$.
\end{proof}

\section{The case of the zeta function}

Let us suppose that there are no restrictions on the digits, i.e.\@ $N=b$.
Then $s_0=1$ and $K(s) = \zeta(s)$. The formula of Proposition
\ref{prop:res} says that
\begin{equation*}
  \lambda_0 = (\log b)^{-1}\lim_{\ell\to\infty} \sum_{b^{\ell-1}\leq n<b^{\ell}}n^{-1},
\end{equation*}
which by a Riemann sum argument using $\int_{b^{-1}}^1 \frac{\dt}t$ gives
$\lambda_0=1$. More generally, at $s=1+k 2\pi i/\log(b)$, $k\neq 0$, we
need to evaluate
\begin{align*}
  \lim_{\ell\to\infty} \sum_{b^{\ell-1}\leq n<b^{\ell}}n^{-1-k2\pi i/\log(b)}
&= \lim_{\ell\to\infty} b^{-\ell(1 + k 2\pi i/\log(b))}
   \sum_{b^{\ell-1}\leq n<b^{\ell}}(b^{-\ell}n)^{-1-k2\pi i/\log(b)}\\
&= \int_{b^{-1}}^1 t^{-1-k 2\pi i/\log b}\dt
=
   \left[\frac{t^{-k2\pi i/\log(b)}}{-k2\pi i/\log(b)}\right]_{b^{-1}}^1 = 0.
\end{align*}
This (fortunately...) confirms that there are no poles apart from $s=1$ on the
line $\Re(s)=1$.  Thus, using Equation \eqref{eq:recres}, the other poles, if
any, are at non-positive integers.

We continue the discussion later, but indulge here in a short digression about
Equations \eqref{eq:res} and \eqref{eq:resH}. They give respectively:
\begin{align*}
  \log b &=  \sum_{1\leq a < b} \frac 1a +
\sum_{m=1}^\infty (-1)^mb^{-m-1}(\sum_{1\leq a<b} a^m) \zeta(m+1),
\\    
    \log b &=  \sum_{1\leq a <b} \frac1{a+1}  
 + \sum_{m=1}^\infty \frac{\sum_{1\leq a<b} a^m}{b^{m+1}}(\zeta(m+1)-1).
\end{align*}
The second one has geometric convergence of ratio $(1-b^{-1})/2$, we can improve
the quality of the first one this way:
\begin{align*}
  \log b &= \sum_{1\leq a < b} \frac 1a +
  \sum_{m=1}^\infty\frac{(-1)^ma^m}{b^{m+1}} + \sum_{m=1}^\infty
  (-1)^mb^{-m-1}(\sum_{1\leq a<b} a^m) (\zeta(m+1) -1)
  \\
  &= \sum_{1\leq a < b} \frac 1a - \sum_{0\leq a<b} \frac{a}{b(a+b)} +
  \sum_{m=1}^\infty (-1)^mb^{-m-1}(\sum_{1\leq a<b} a^m) (\zeta(m+1) -1)
  \\
  &= \sum_{1\leq a < 2b}\frac1a -1 + \sum_{m=1}^\infty
  (-1)^m\frac{\sum_{1\leq a<b} a^m}{b^{m+1}} (\zeta(m+1) -1).
\end{align*}
We can average the two representations of $\log b$ to have a sum involving
only odd zeta values. Here is with $b=2$:
\begin{align*}
  \log 2 &= \dfrac56 - \sum_{m=1}^\infty (-1)^{m-1}\dfrac{\zeta(m+1) -1}{2^{m+1}}
\\
 &=  \dfrac12 
 + \sum_{m=1}^\infty \dfrac{\zeta(m+1)-1}{2^{m+1}},
\\
&= \frac23 + \frac12\sum_{p=1}^\infty \frac{\zeta(2p+1)-1}{4^p}.
\end{align*}
The last formula (average of the first two) is geometric with ratio $1/16$. A
slower equivalent form is $\frac23+\frac12\sum_{n=2}^\infty
(n(4n^2-1))^{-1}$.

Let us examine again \eqref{eq:Kmain} for the Riemann zeta function (and using
$\ell=1$).  So $N=b$ and the formula reads:
\begin{equation}\label{eq:zeta}
  \bigl(1 - b^{1-s}\bigr) \zeta(s) =
  \sum_{1\leq a<b} a^{-s} 
  + \sum_{m=1}^\infty (-1)^m\frac{(s)_m}{m!} b^{-s-m}\gamma_m \zeta(s+m).
\end{equation}
Due to $(s)_m = s(s+1)\dots(s+m-1)$, the series extended from $m=2$ to
$+\infty$ (which is absolutely convergent for $\Re s>-1$) vanishes at $s=0$.
And the term with $m=1$ gives a contribution equal to $-b^{-1}\gamma_1=-(b-1)/2$
as we already know that the residue is $1$ at $s=1$. So, there is no
pole but a finite value at $s=0$:
\begin{equation*}
  (1-b)\zeta(0) = b-1 - \frac{b-1}2 = \frac{b-1}2\implies \zeta(0)=-\frac12.
\end{equation*}
Looking at Equation \eqref{eq:zeta} for $-\frac32<\Re s < -\frac12$, the terme
with $m=1$ (evaluation at $s+1$) gives no pole, and $m=2$ does not either as
the numerator vanishes at $s=-1$.  Proceeding inductively to $\frac12 -m <\Re
s < \frac12 - m+1$ for $m>2$ to check if there is a pole at $s=1-m$, we again
find that the $m$-th term of the series actually contributes no pole because
the numerator vanishes.  So the zeta function has only a single pole in the
complex plane, located at $s=1$.

\section{Residues (II)}

We return to the general case.
\begin{prop}\label{prop:recres}
  Let $\lambda_{m,k}$ be the residue (perhaps vanishing) of $K(s)$ at
  $s_{m,k}=s_0 - m + k 2\pi i /\log b$.  The following relations hold for
  $m\geq1$:
  \begin{equation}\label{eq:recres}
    (1 - b^{-m})\lambda_{m,k} =
    -N^{-1}\sum_{j=1}^m\frac{(m-s_{0,k})\dots
                            (m-j+1-s_{0,k})}{j!}
    \frac{\gamma_j}{b^j} \lambda_{m-j,k}.
  \end{equation}
  If $\lambda_{0,k}=0$, i.e.\@ if $s_{0,k}$ is not a pole, then $s_{m,k}$
  is not a pole either, for every $m\geq1$.
\end{prop}
\begin{proof}
  The linear recurrence is a consequence of \eqref{eq:Kmain} (details left to
  reader; it is recommended to first replace the letter $m$ used as summation
  index in \eqref{eq:Kmain} by the letter $j$).  And induction using
  \eqref{eq:recres} then shows that if $\lambda_{0,k}=0$ then $\lambda_{m,k}$
  vanishes for every $m\geq1$.
\end{proof}
The equation \eqref{eq:recres} is also valid for $m=0$ as the empty sum on the
right-hand side has value zero.

We define new quantities $\mu_{m,k}$ such that the following relation holds:
\begin{equation}\label{eq:defmu}
  \lambda_{m,k} = \frac{-(-s_{0,k})_{m+1}}{m!} \mu_{m,k} =
(-1)^{m}\frac{s_{m,k}(s_{m,k}+1)\dots (s_{m,k}+m)}{m!} \mu_{m,k}. 
\end{equation}
But this definition fails for $k=0$ if $(-s_0)_{m+1}$ can vanish, i.e.\@ if
$s_0$ is $0$ (case $N=1$) or $1$ (case $N=b$).  We will define later
$\mu_{m,0}$ if $N=1$ and $N=b$ in an \emph{ad hoc} manner, in order for the next
Proposition to hold for them too.  For $1<N<b$ one has in particular:
\begin{equation*}
  \mu_{0,k} = \frac{\lambda_{0,k}}{s_{0,k}},\quad
  \mu_{1,k} = \frac{-\lambda_{1,k}}{s_{0,k}(s_{0,k}-1)}.
\end{equation*}
\begin{prop}\label{prop:recmu}
The following recurrence formula holds, for $m\geq1$:
\begin{equation}\label{eq:recmu}
  (1 - b^{-m})\mu_{m,k} = 
  -N^{-1}\sum_{j=1}^m\binom{m}{j}b^{-j}\gamma_j \mu_{m-j,k}\,.
\end{equation}
It can be written in the alternative form, with $\mu_{m,k}$ on both sides:
\begin{equation}\label{eq:recmubis}
  \mu_{m,k} = \frac{b^m}{N}\sum_{j=0}^m \binom{m}{j}\frac{\gamma_j}{b^j} \mu_{m-j,k}\,.
\end{equation}
\end{prop}
\begin{rema}
  For $N=1$, and $k=0$, equation \eqref{eq:defmu} can not be used
  because $s_0=0$.  For $k\neq 0$ one checks via a simple computation
  from Equation \eqref{eq:resN1}, which gives the value of $\lambda_{m,k}$,
   that 
  \begin{equation*}
    \mu_{m,k} =  \frac{(-1)^m}{2\pi i k }\qquad (N=1; k\neq0).
  \end{equation*}
  In order for Proposition \ref{prop:recmu} to hold also for $(\mu_{m,0})$,
  i.e.\@ for the recurrence \eqref{eq:recmu} to be verified, we can choose any
  constant $c$ and set $\mu_{m,0}=(-1)^mc$ because we know that this
  recurrence is verified by $((-1)^m(2\pi i k)^{-1})_{m\geq0}$, for example
  with $k=1$. We choose $\mu_{m,0} = (-1)^m$.

  For the Riemann zeta function, we define $(\mu_{m,0})$ to be the solution of
  \eqref{eq:recmu} such that $\mu_{0,0} = 1$.  It turns out that then
  $\mu_{m,0} = B_m$, the $m$-th Bernoulli number (with $B_1 = -\frac12$).  Indeed,
  the recurrence \eqref{eq:recmubis}, in this $N=b$ case,
  is solved by the Bernoulli numbers due to the identity
  $(\e^t-1)^{-1} \sum_{a=0}^{b-1}\exp(b^{-1}a t) = (\e^{b^{-1}t}-1)^{-1}$.
\end{rema}
\begin{proof}[Proof of Proposition \ref{prop:recmu}]
  The details of the deduction of Equation \eqref{eq:recmu} from Equations
  \eqref{eq:recres} and \eqref{eq:defmu}, if $1<N<b$, or if $N$ is arbitrary
  but $k$ is non-zero, are left to the reader.  In the case $N=1$ and $N=b$
  with $k=0$, the equation \eqref{eq:recmu} is true due to the \emph{ad hoc}
  definition of the sequence $(\mu_{m,0})$. The formula of \eqref{eq:recmu}
  holds also with $m=0$, as an empty finite sum is defined to be zero.  One
  uses $\gamma_0=N$ to relate the two ways of writing the recurrence.
\end{proof}
An immediate observation relative to \eqref{eq:recmu} is that when we vary $k$
we keep the exact same recurrence relation.  Hence:
\begin{prop}\label{prop:mukmu0}
  For every $m\in\NN$ and $k\in \ZZ$ there holds:
  \begin{equation}
    \mu_{m,k} = \frac{\mu_{0,k}}{\mu_{0,0}}\mu_{m,0}.
  \end{equation}
\end{prop}
In other terms, once we know the residues on the horizontal axis, we don't
need to solve again \eqref{eq:recres} for $k\neq0$, if we know
$\lambda_{0,k}$.

Let us now consider the (formal, at this stage) normalized exponential
generating function for the modified residues on the horizontal axis:
\begin{equation*}
  B(t) = (\mu_{0,0})^{-1}\sum_{m=0}^\infty\frac{\mu_{m,0}}{m!}t^m.
\end{equation*}
We compute the product:
\begin{align*}
  (\sum_{a\in A} \e^{b^{-1}at}) B(t) &= \sum_{j=0}^\infty \frac{b^{-j}\gamma_j}{j!}t^j
                                    \sum_{q=0}^\infty\frac{\mu_{q,0}}{q!}t^q
\\
&
=\sum_{m=0}^\infty (\sum_{j=0}^m \binom{m}{j}b^{-j}\gamma_j\mu_{m-j,0})\frac{t^m}{m!}
\\
&=
\sum_{m=0}^\infty Nb^{-m}\mu_{m,0}\frac{t^m}{m!} 
= N B(b^{-1}t).
\end{align*}
This gives the functional equation:
\begin{equation}
  \label{eq:genfunc}
  B(t) = \frac{N}{\sum_{a\in A} \e^{b^{-1}at}} B(b^{-1}t).
\end{equation}
We can now state the following somewhat striking corollary:
\begin{theo}\label{thm:E}
  Let $E(t)$ be the moment generating function for the canonical probability
  measure with support the Cantor set of those real numbers in $\Iff01$ whose
  infinite $b$-radix expansion (or at least one of them, if there are two)
  uses admissible digits only.  The exponential generating function
  $B(t)=(\mu_{0,0})^{-1}\sum_{m=0}^\infty\frac{\mu_{m,0}}{m!}t^m$ of the
  normalized residues on the real axis of the restricted Dirichlet series
  $K(s)$ is its multiplicative inverse:
  \begin{equation}
    B(t) = \lim_{k\to\infty}\frac{N^k}
                 {\sum\limits_{a_1, \dots, a_k \in A}\e^{t\sum_{j=1}^k b^{-j}a_j}}=
\Bigl(\prod_{j=1}^\infty \frac{\sum_{a\in A} \e^{b^{-j}at}}{N}\Bigr)^{-1}\;.
  \end{equation}
  It is a meromorphic function in the entire complex plane, as inverse of
  an entire function. It verifies $\Re t<0 \implies |B(t)|>1$.
\end{theo}
\begin{rema}
  The reference to poles on the real axis is a bit vague, as it allows for a
  vanishing residue. The precise formulation is given by the formula defining
  $B(t)$ and the prior definitions.  When we refer to $b$-radix expansion, we
  mean an infinite one: so if $0$ is not an admissible digit, and $b-1$ isn't
  either, ``$b$-decimal'' numbers are not in the alluded-to Cantor set even if
  their finite expansion uses only digits from $A$.  And of course, in the
  case of the Riemann zeta function, the ``Cantor'' set is the full unit
  interval $\Iff01$.
\end{rema}
\begin{rema}
  In the case $b=N$, which corresponds to the Riemann zeta function
  $\zeta(s)$, the function $E(t)$ is $\int_0^1\e^{tx}\dx = \frac{\e^t-1}{t}$.
  Hence $B(t) = \frac{t}{e^t-1}$ is indeed the exponential generating function
  of the $\mu_{m,0}/\mu_{0,0}$ quantities in that case, as we have argued
  earlier that $\mu_{m,0}=B_m$ due to their \emph{ad hoc} definition as the solution
  of Equation \eqref{eq:recmubis} with $\mu_{0,0}=1$.  This makes, in
  general, the ratios $\mu_{m,0}/\mu_{0,0}$ (which are rational numbers)
  some sort of generalized Bernoulli numbers associated with $b>1$ and
  $A\subset \Iffint{0}{b-1}$, $A\neq\emptyset$, $A\neq\{0\}$.
\end{rema}
\begin{proof}[Proof of Theorem \ref{thm:E}]
  We first show that the infinite product
\[
  E(t)=\prod_{j=1}^\infty \frac{\sum_{a\in A} \e^{b^{-j}at}}N
\]
converges everywhere in the complex plane and defines an entire function.  For
$N=1$ and $f>0$ the allowed digit, this clearly is the case with
$E(t)=\exp(\frac{f}{b-1}t)$.  If $N>1$, it is also clearly the case from
the fact that
\begin{equation*}
  \frac{\sum_{a\in A} \e^{b^{-j}at}}N = 1 + O_{t\to0}(t).
\end{equation*}
And, for any bounded region, only finitely many terms may have zeros in it,
and removing those terms we get a uniformly convergent non-zero
product. Zeros of the entire function $E(t)$ are of the shape $b^jz$,
$j\geq1$, where $z$ is in finitely many additive cosets of $2\pi i \ZZ$,
corresponding to $\e^z$ being a non-zero root of the polynomial
$P_A(z)=\sum_{a\in A} z^a$.  The function $E$ is of finite
exponential type $f/(b-1)$ but we don't need this. For $\Re t < 0$, each term
of the product is of complex modulus less than $1$, so $|E(t)|<1$.

Let $r>0$ be the smallest modulus of a root of $E$ (or $r=\infty$ for
$N=1$), and consider the power series expansion of $B(t) = E(t)^{-1}$:
\begin{equation*}
  B(t) = 1
  + \sum_{j=1}^\infty \frac{c_j}{j!} t^j,
\end{equation*}
which exists and is convergent for $|t|<r$. We set $c_0=1$. By
construction of $E(t)$ there holds
\begin{equation*}
  B(t) = \frac{N}{\sum_{a\in A} \e^{b^{-1}at}} B(b^{-1}t).
\end{equation*}
Reversing the steps of an earlier calculation this means that the sequence
$(c_m)_{m\geq0}$ shares the same equivalent recurrence formulas \eqref{eq:recmu}
and \eqref{eq:recmubis}
as does the sequence $(\mu_{m,0}/\mu_{0,0})_{m\geq0}$. As they have the same
starting point at $m=0$, this establishes that $c_m = \mu_{m,0}/\mu_{0,0}$.

To complete the proof we explain briefly the connection with Cantor sets.
Consider infinitely many identically distributed and independent random
variables $X_j$, $j\geq1$, each probability $1/N$ to
take any given value $a\in A$.  It is well-known that, if $\Omega$ is the
probability space, and $x:\Omega\to\Iff01$ is the map
\begin{equation*}
  x(\omega) = \sum_{j=1}^\infty \frac{X_j(\omega)}{b^j},
\end{equation*}
(which is maybe not one-to-one if $A$ contains both $0$ and $b-1$, but this is
benign) then the push-forward under $x$ of the probability measure to $\Iff01$
gives a Borel measure, which is Lebesgue measure if $N=1$, else is singular to
Lebesgue measure and supported by the Cantor set whose first construction step
is to replace $\Iff01$ by the union of the intervals
$\Iff{b^{-1}a}{b^{-1}(a+1)}$, $a\in A$, then repeat the operation in each of
them, etc\dots.  For any complex $t$, the expectation value of the random
variable $\e^{tx}$ on $\Omega$ becomes the integral of the function $x\mapsto
\e^{tx}$ against that push-forward probability measure.  If we compute it on
$\Omega$ as the limit from approximation by simple functions constant on
cylinder sets defined by finitely many conditions $X_1=a_1$, \dots, $X_k =
a_k$, we obtain that this expansion is the limit as $k$ goes to infinity of
\begin{equation*}
  N^{-k}\sum\limits_{a_1, \dots, a_k \in A}\e^{t\sum_{j=1}^k b^{-j}a_k}= 
   \prod_{j=1}^k\frac{\sum_{a\in A} \e^{b^{-j}at}}{N}.
\end{equation*}
In other terms it is exactly $E(t)$.
\end{proof}

\section{Values at non-positive integers}

In this final section we study the values of $K(s)$ at the non-positive integers.
We shall assume $N>1$, as the non-positive integers are poles if $N=1$.  We
also assume $N<b$, as the pole at $s=1$ of the Riemann zeta function modifies
the details of the discussion.

Consider Equation \eqref{eq:Kmain} and substitute $s=0$.  All contributions of
the series on the right-hand side of this equation vanish.
It follows:
\begin{equation*}
  K(0) = \frac{-N_1}{N-1}.
\end{equation*}
We will discover that it is important whether $N_1=N$ or $N_1=N-1$.
\begin{prop}\label{prop:zerovalues}
  The sequence $(K(-m))_{m\geq0}$, assuming $1<N<b$, consists of rational
  numbers verifying $K(0)=-N_1/(N-1)$ and the following recurrence:
  \begin{equation}
    \label{eq:kmrec}
    m\geq1
\implies
    (b^{-m} - N)K(-m) = \frac{\gamma_m}{b^m} + 
      \sum_{j=1}^m \binom{m}{j}\frac{\gamma_j}{b^j}K(-m+j).
  \end{equation}
  If $0$ is an admissible digit, i.e.\@ if $K(0)=-1$, then $K(-m)=0$ for every
  $m\geq1$.  If $0$ is not an admissible digit, i.e.\@ if $N_1 = N$, equivalently if
  $K(0)<-1$, the exponential generating function $C(t) = \sum_{m=0}^\infty
  K(-m)\frac{t^m}{m!}$ verifies the following functional equation:
  \begin{equation}\label{eq:kmgen}
    C(b^{-1}t) = \alpha_A(b^{-1}t) + \alpha_A(b^{-1}t) C(t),
  \end{equation}
  where $\alpha_A(t) = \sum_{a\in A}\e^{at}$. It has a positive radius of convergence
  and is a meromorphic function in the entire complex plane. It is given ``explicitly'' as
  \begin{equation}
    \label{eq:C}
    C(t)=\sum_{m=0}^\infty
    K(-m)\frac{t^m}{m!} = -1 - \sum_{k=1}^\infty\frac1{\prod_{j=1}^k \alpha_A(b^{-j}t)}\,.
  \end{equation}
  At every location except at the elements of $\{b^j z, j\geq1, \alpha_A(z)=0\}$, the
  series at the right-hand side of \eqref{eq:C} converges geometrically.
\end{prop}
\begin{proof}
  Equation \eqref{eq:kmrec} is a corollary to \eqref{eq:Kmain} (details left
  to reader; it is recommended though to replace letter $m$ in
  \eqref{eq:Kmain} by letter $j$ at the very start).  As $K(0)$ is rational,
  all the $K(-m)$'s are, too. The equation determining $K(-1)$ is
  \begin{equation*}
    (b^{-1} -N)K(-1)=b^{-1}\gamma_1 + b^{-1}\gamma_1 K(0).
  \end{equation*}
So, if $K(0)=-1$, we get $K(-1)=0$.  But then 
  \begin{equation*}
    (b^{-2} -N)K(-2)=b^{-2}\gamma_2+ b^{-2}\gamma_2 K(0) = 0.
  \end{equation*}
And by induction $K(-m)=0$ for every $m\geq1$.

We now suppose that $N_1\neq N-1$, i.e.\@ $0\notin A$ and $N_1=N$. 
We rewrite \eqref{eq:kmrec} in the equivalent form (as $\gamma_0=N$):
\begin{equation*}
   b^{-m}K(-m) = \frac{\gamma_m}{b^m} + 
      \sum_{j=0}^m \binom{m}{j}\frac{\gamma_j}{b^j}K(-m+j).
  \end{equation*}
  Observe that this equation holds true also for $m=0$. This is due to
  $\gamma_0=N$ and $N=N_1$. So, the formal power series $C(t)$ verifies:
  \begin{equation*}
    C(b^{-1}t) = \sum_{m=0}^\infty \frac{\gamma_m}{b^m}\frac{t^m}{m!} 
 + C(t)\sum_{m=0}^\infty \frac{\gamma_m}{b^m}\frac{t^m}{m!}.
  \end{equation*}
  As $\sum_{m=0}^\infty \gamma_m\frac{t^m}{m!} = \alpha_A(t)$ we recognize the
  functional equation \eqref{eq:kmgen}. Conversely, this identity of power
  series leads back to the recurrence (this again crucially uses $N=N_1$; we
  can write a functional equation in the general case, but as its solution if
  $N_1<N$ is already known we skip that).

  Let us now directly examine the expression given in the right-hand side of
  \eqref{eq:C}.  The denominators vanish at the points of the set $\{b^jz,
  j\geq1, \alpha_A(z)=0\}$ whose intersections with bounded regions are
  finite.  Take any large closed disk and consider the right-hand side of
  \eqref{eq:C} but removing from it all occurrences of those finitely many
  functions $\alpha_A(b^{-1}t)$, \dots, $\alpha_A(b^{-k}t), \dots$ having
  zeros in this disk.  The remaining expression is then everywhere defined,
  and as $\alpha_A(b^{-j}t)$ converges uniformly on such bounded region toward
  $N>1$ as $j\to \infty$, the series converges geometrically, and defines an
  analytic function.  Restoring the removed factors we obtain that the
  expression converges everywhere away from the poles and defines a
  meromorphic function $D(t)$.  Let $r>0$ be the smallest modulus of a root of
  $\alpha_A(b^{-1}t)$, we can expand $D(t)$ for $|t|<r$ into a convergent
  power series $\sum_{m=0}^\infty d_m t^m/m!$.  The function $D(t)$ verifies
  the functional equation \eqref{eq:kmgen}.  Note in passing that this is not
  an homogeneous equation and the value $D(0)$ is constrained by it, but we
  can also compute it from the series as $-\sum_{k=0}^\infty N^{-k} = -(1 -
  1/N)^{-1}=-N/(N-1)=K(0)$.  The coefficients $(d_m)_{m\geq0}$ verify the same
  recurrence and are thus identical with the $(K(-m))_{m\geq0}$.  This
  completes the proof.
\end{proof}



\providecommand\bibcommenthead{}
\def\blocation#1{\unskip} \footnotesize



\end{document}